\documentclass[11pt, leqno, 
oneside]{amsart} 

\usepackage{amsmath,amsthm,amssymb} 
\usepackage{amsfonts} 
\usepackage{txfonts} 
\usepackage{thmtools} 

\usepackage{geometry} 

\usepackage[dvipsnames, svgnames]{xcolor} 

\usepackage[shortlabels]{enumitem} 

\usepackage[colorlinks,citecolor=blue,hypertexnames=false,urlcolor=blue,breaklinks=true]{hyperref} 
\newcommand{\Lim}[1]{{\displaystyle \lim_{ #1}}}
\newcommand{\Int}[2]{{\displaystyle \int_{ #1}^{ #2}}}
\newcommand{\Frac}[2]{\displaystyle{\frac{\displaystyle{#1}}{\displaystyle{#2}}}}
\newcommand{\pde}[2]{{\displaystyle \frac{\mbox{$\partial #1$}}{\mbox{$\partial #2$}}}} 

\usepackage{accents} 
\usepackage{exscale} 
\usepackage{dsfont} 
\usepackage{mathrsfs} 
\usepackage[T1]{fontenc} 

\usepackage{etoolbox}
\appto\appendix{\addtocontents{toc}{\protect\setcounter{tocdepth}{1}}}

\newcommand{\be}{\begin{equation}}
	\newcommand{\ba}{\begin{array}}
		\newcommand{\ea}{\end{array}}
\newcommand{\eeq}[1]{\label{eq:#1}\end{equation}}
\newcommand{\ET}[1]{\end{sl}\label{theorem:#1}\end{theorem}}
\newcommand{\Bt}{\begin{theorem}\begin{sl}}
\newcommand{\Eqref}[1]{{\rm (\ref{eq:#1})}}
 
\newcommand{\EL}[1]{\end{sl}\label{lemm:#1}\end{lemm}}
\numberwithin{equation}{section} 
\newcommand{\cald}{\mathcal D}

\theoremstyle{plain}
\newtheorem{theorem}{Theorem}[section]
\newtheorem{lemma}{Lemma}[section]

\theoremstyle{definition}

\theoremstyle{remark}
\newtheorem{remark}[theorem]{Remark}

\newtheorem{claim*}{Claim}

\newcommand{\bb}[1]{\mathbb{#1}}

\newcommand{\dist}{\operatorname{dist}}

\newcommand{\ra}{\rightarrow}
   
\newcommand{\loc}{\operatorname{loc}}     

\DeclareMathOperator{\supp}{supp}

\newcommand{\Div}{\mbox{\rm div}\,} 
\newcommand{\R}{\mathbb R}

\newcommand{\curl}{\mbox{curl}\,}

\def\XXint#1#2#3{{\setbox0=\hbox{$#1{#2#3}{\int}$}
\vcenter{\hbox{$#2#3$}}\kern-.5\wd0}}

\def\YYint#1#2#3{{\setbox0=\hbox{$#1{#2#3}{\iint}$}
\vcenter{\hbox{$#2#3$}}\kern-.51\wd0}}
 
\allowdisplaybreaks

\setlist{nosep}

\begin{document}

\author[R. M. Chen]{Robin Ming Chen}
\email[]{mingchen@pitt.edu}
\address{Department of Mathematics, University of Pittsburgh, Pittsburgh, PA, 15260}

\author[G. P. Galdi]{Giovanni P. Galdi}
\email[]{galdi@pitt.edu}
\address{Department of Mechanical Engineering and Materials Science, University of Pittsburgh, Pittsburgh, PA, 15261}

\author[B. Poggi]{Bruno Poggi}
\email[]{brunopoggi@pitt.edu}
\address{Department of Mathematics, University of Pittsburgh, Pittsburgh, PA, 15260}

\author[A. Schikorra]{Armin Schikorra}
\email[]{armin@pitt.edu}
\address{Department of Mathematics, University of Pittsburgh, Pittsburgh, PA, 15260}

\newcommand{\eps}{\varepsilon}
\newcommand{\aleq}{\precsim}
\newcommand{\ageq}{\succsim}
\newcommand{\aeq}{\asymp}
\newcommand{\Ds}[1]{|D|^{#1}}

\title[On Serrin Criterion ]{On Serrin Interior Regularity Criterion\\ for Navier-Stokes Equations}

\begin{abstract} We revisit Serrin's interior spatial regularity criterion for distributional solutions to the Navier-Stokes equations in $\mathbb R^3$ and considerably relax the hypotheses in two main directions. More precisely, we show that if $u\in{L_t^{s'}L_x^s}$ locally is a distributional solution to the Navier-Stokes equations with $\frac2{s'}+\frac3s=1$ for $s'\in[4,\infty)$, then $u\in L^q_t(C_x^\infty)$ locally for all $q\in(2,s')$. If $s'\in(2,4)$, the same conclusion holds provided that in addition $u\in L_t^4(L_x^p)$ locally, for some $p>1$. 

In particular, we remove any integrability hypothesis on the vorticity, and we reduce the requirement of integrability in time all the way to $L^4$ from $L^\infty$. To achieve this, we employ a new bootstrap argument, distinct from Serrin's, and we argue that a reduction of the exponent in time integrability does not follow from Serrin's original argument.
\end{abstract}

\maketitle

\hypersetup{linkcolor=blue}

\tableofcontents

\section{Introduction} As is well known, Serrin's regularity criterion \cite{Serrin} constitutes a cornerstone in the theory of the Navier-Stokes equations. Precisely, let $Q=B\times(0,T)$, with $B$ a ball in $\mathbb{R}^3$, be a space--time cylinder, and let $\mathcal{D}(Q)$ be the space of solenoidal functions in $C_0^\infty(Q)$. A vector field $u\in L^2(Q)$,  is a distributional solution to the Navier-Stokes equations if
\begin{equation}\ba{ll}\medskip\displaystyle
\int_Q u\cdot(\varphi_t+\Delta \varphi+u\cdot\nabla\varphi)\,{\rm d}x{\rm d}t=0\,,\ \ \mbox{for all $\varphi\in \mathcal D(Q)$}\,,\\\displaystyle
\int_Bu\cdot\nabla\phi\, {\rm d}x=0\,,\ \ \mbox{for all $\phi\in C_0^\infty(B)$\,.}
\ea
\end{equation}
Then, in 3D, the criterion states that if $u$  is any distributional solution such that\footnote{As usual, $L^p(X)$ denotes $L^p$-integrability in time with values in $X$. When $X=L^q$, we may also use the notation $L^{p,q}$. For further notation, see Remark \ref{rm.notation}.}
\be
u\in L^\infty(L^2(B))\,,\ \ \omega:=\curl u\in L^{2}(Q)
\eeq{2}
and, in addition,
\be
u\in L^{s'}(L^s(B))\,,\ \ \ \frac2{s'}+\frac3s=1\,,\ \ s'\in(2,\infty)\,,
\eeq{3}
then $u$ is $C^\infty$ in the spatial variables and every spatial derivative of $u$ belongs to $L^\infty(Q')$ for any $Q'\Subset Q$. For brevity, we denote this property by
\be
u\in L^\infty\big(C^\infty\big)\,. 
\eeq{4}

Actually, Serrin proved \Eqref{4} under the stronger assumption that the exponents $s',s$ satisfy the relation above with a strict inequality sign. The proof of the result, as stated here, is due to Struwe \cite{Str} and Takahashi \cite{Taka}. 
\par
It is worth emphasizing that \Eqref{4} yields regularity {\em only} in the spatial variables; no gain in time regularity is obtained beyond that assumed for $u$ in \Eqref{2}. This feature is, in general, sharp. Indeed, Serrin's celebrated example \cite{Serrin} shows that the function
$$
u=a(t)\nabla\psi, \qquad \Delta\psi=0 \ \text{in}\ B,
$$
with $a\in L^2(0,T)$, is a distributional solution that is $C^\infty$ in space, whereas its regularity in time is exactly that of the coefficient $a$. 
\par
However, this same example brings up an interesting question, namely,  whether the spatial regularity provided by \Eqref{4} persists if the first assumption on $u$ in \Eqref{2} is relaxed to $u\in L^q(L^2(B))$ for some $q\in(2,\infty)$. In other words, with the obvious meaning of the notation  (see Remark \ref{rm.notation}), does it still follow that $u\in L^q(C^\infty)$?\footnote{We thank Professor Igor Kukavica for bringing this problem to our attention.} Interestingly, Serrin's method does not appear to be capable of yielding a positive answer. In fact, the assumption that $u\in L^\infty(L^2(B))$ is an essential ingredient of his argument. To clarify this point, let us briefly review the main steps of the proof. It is based on a boot-strap procedure that uses the following two basic properties. First, the (local) Helmholtz decomposition of $u$:
\be   
u(x,t)=\int_B\nabla \mathscr E(x-y)\times\omega(y,t)\,{\rm d}y +\Phi(x,t)\ \ \mbox{in $Q$}\,, 
\eeq{5}
with $\mathscr E$ Laplace's fundamental solution and $\Phi\in C^\infty(B)$\,. Moreover, observing that $\omega$ satisfies, in the distributional sense, the equation
\be
\omega_t-\Delta\omega=\Div(u \otimes \omega - \omega \otimes u)\ \ \mbox{in $Q$}
\eeq{om}
one gets
the integral representation:
\be  
\omega(x,t)=\nabla K\star (u \otimes \omega - \omega \otimes u)+S(x,t)\ \ \mbox{in $Q$}
\eeq{6}
where $K$ is the heat kernel and $S$ is a solution to the heat equation.
Now, assuming only  the second assumption on $\omega$ in \Eqref{2}  and \Eqref{3}, one shows that
$$
\omega\in L^\beta(L^\infty)
$$
for every $\beta\in(1,\infty)$; see \cite{Serrin,Taka}. Then, plugging this information into \Eqref{5} and supposing that $u$ is $L^q$ in time, for some $q\in(2,\infty)$, entails $u \otimes \omega - \omega \otimes u=:g\in L^q(L^\infty)$. Consequently, recalling that $q>2$, \Eqref{6} and classical regularity theory imply that 
$$
\omega\in L^\infty(C^\alpha)\,,\ \ \alpha\in (0,1)\,.
$$
In view of \Eqref{5}, this furnishes
$$
u\in  L^q(C^{1,\alpha})\,,
$$
which implies
$$
g\in L^q(C^\alpha)\,.
$$
At this point, to boot-strap the argument with the {\em same} $q$, we should use the latter into \Eqref{6} and hope to  obtain
$$
\omega\in L^\infty(C^{1,\alpha})\,.
$$
However, from well-known results on the heat equation (e.g. \cite[Theorem 3.1]{Krylov}),
we only get
$$
\omega\in L^q(C^{1,\alpha})\,,
$$ 
which, once used in \Eqref{5}, furnishes
$$
u\in L^q(C^{2,\alpha})\,.
$$
Therefore,
$$
g\in L^{\frac q2}(C^{1,\alpha})
$$
which then gives
$$
\omega \in L^{\frac q2}(C^{2,\alpha})\,, \ \ \mbox{if\, $q/2>2$}\,,
$$
and so, by \Eqref{5},
$$
u \in L^{\frac q2}(C^{3,\alpha})\,, \ \ \mbox{if \,$q>4$}\,,
$$
and so forth. The foregoing argument reveals that improving the spatial regularity of $u$ demands a parallel improvement in its time regularity. While this poses no difficulty under Serrin's assumption $u\in L^\infty$, it breaks down when one merely assumes $u\in L^q$ for a {\em fixed} $q<\infty$, no matter how large $q$ is.
\par
The aim of this note is to employ a bootstrap argument different from Serrin's and thereby answer the foregoing question in the affirmative. In fact, we establish a stronger result. More precisely, we show that the assumption \Eqref{2}$_2$ can be substantially weakened and, for $s'$ in a suitable range, can even be dispensed with altogether. Our main theorem reads as follows.

\begin{theorem}\label{theo:maingen}
Let $u$ be a distributional solution to the Navier-Stokes equations such that
\be
u\in L^{s'}(L^s(B))\,,\ \ 
2/s'+3/s=1\,,\ \ s'\in (2,\infty)\,.
\eeq{B11}
Then, if $s'\in[4,\infty)$,
$$
u\in L^q(C^\infty) \quad   \mbox{ for all } \ q \in (2, s'). 
$$
If $s'\in (2,4)$, the same conclusion holds, provided that, in addition to \Eqref{B11}, $u\in L^{4}(L^p(B))$ for some $p>1$.
\end{theorem}
Our boot-strap argument goes as follows. We begin by showing that
\[
\omega\in L^q(H^{1,r}),
\qquad r\in(1,\infty),
\]
which, by \Eqref{5}, implies
\[
u\in L^q(H^{2,r}).
\]
As a consequence, 
\[
F:=\Div(u \otimes \omega - \omega \otimes u)\in L^{q/2}(L^r).
\]
We then exploit parabolic regularity for \Eqref{om} in Bessel potential spaces to show that, for sufficiently large $r$,
\[
\omega\in L^q(H^{\ell_1,r})
\]
for some  $\ell=\ell(q)>1$, while preserving the {\em same} time integrability exponent $q$. It follows again from \Eqref{5} that
\[
u\in L^q(H^{\ell +1,r})\,,
\]
which, in turn, yields
\[
F\in L^{q/2}(H^{k,r})\,,\ \   k:=\ell_1-1>0.
\]
Using this information in \Eqref{om}, we then deduce, at the next step, 
$$
\omega\in L^q(H^{\ell + k,r})
$$
Iterating this argument, we obtain eventually
\[
\omega\in L^q(C^\infty),
\]
and consequently,
\[
u\in L^q(C^\infty).
\]   
As regards the assumption $\omega\in L^2(Q)$, we prove that it follows from \Eqref{B11}. Indeed, by adapting the argument of \cite{GAr}, one obtains $\omega\in L^2(Q)$ for all $s'\in[4,\infty)$, while for $s'\in(2,4)$ the same conclusion holds provided that $u\in L^4(L^p(B))$, for some $p>1$.

The remainder of the paper is organized as follows. In Section~2 we prove Theorem~\ref{theo:maingen} under the additional assumption $\omega\in L^2(Q)$. In Section~3 we show that this assumption can be removed under the above conditions on $s'$.

\begin{remark}\label{rm.notation} Let us make a remark about notation. Let $X$ be a function space defined over open sets $U$ of $\bb R^3$, and $r\in[1,\infty]$. We say $f\in X_{\loc}(U)$ if $f\in X(U')$ for all open $U'$ compactly contained in $U$. For a vector-valued function $g:(T_1,T_2)\ra X_{\loc}$, we denote 
\begin{equation}\label{eq.local}
g\in L^r(X)
\end{equation}
to mean that $g\in L^r_{\loc}((T_1,T_2);X_{\loc})$. Hence note that the notation (\ref{eq.local}) implies only a control over compactly contained cylinders in $U\times(T_1,T_2)$, and is therefore different from the `global' control $g\in L^r((T_1,T_2);X(U))$. 

If $X=L^q$, then we may write $g\in L^{r,q}$ instead of $g\in L^r(L^q)$.
\end{remark}

\setcounter{equation}{0}

\section{A First Regularity Result}\label{sec.general}
We begin  to show Theorem \ref{theo:maingen} under the extra assumption $\omega\in L^2(Q)$. This assumption will be removed in the next section, thus completing the proof of the theorem. 
\begin{theorem}
\label{theo:02} Let $u$ be a distributional solution to the Navier-Stokes equations such that
\be
u\in L^{s',s}\,,\ \ 
2/s'+3/s=1\,,\ \ s'\in (2,\infty)\,,\ \ \omega\in L^2(Q)\,.
\eeq{B111}
Then, 
$$
u\in L^q(C^\infty) \quad  \mbox{ for all } \ q \in (2, s'). 
$$
\end{theorem}
\noindent
The proof of this result requires the following preparatory lemma. 
\begin{lemma}\label{lem:Bes}  
Consider the Cauchy problem 
\be
\partial_tv-\Delta v=F\,, \ \ \mbox{in $\R^3\times(0,T)$}\,;\ \ v(\cdot,0)=0\,.\eeq{eq:CP}
Let
\be
F\in L^\sigma((0,T); H^{k,r}(\mathbb R^3))\,, \ \ \sigma,r\in (1,\infty)\,,\ \ k\in[0,\infty)\,, \ \ \ell\in(0,2),
\eeq{B2}
and suppose that 
\[
\frac1p:=\frac1{\sigma}+\frac{\ell}2-1>0.
\]
Then there is a solution $v$ such that
\be
v\in L^p((0,T); H^{k+\ell,r}(\mathbb R^3))
\eeq{B3}
and satisfying the following inequality
\be
\int_0^T\|v(\tau)\|_{k+\ell,r}^p{\rm d} \tau \le C\left(\int_0^T\|F(\tau)\|_{k,r}^\sigma{\rm d} \tau \right)^\frac{p}\sigma,
\eeq{St}
where $C$ depends on $r$, $k$, $\ell$, and $\sigma$.
\end{lemma} 

\begin{proof} 
Let $\{F_n\}\subset C_0^\infty(\R^3\times(0,T))$ be a sequence approaching $F$ in $L^\sigma(H^{k,r})$ and set
\be
v_n(t)=\int_0^t{\rm e}^{(t-\tau)\Delta}F_n(\tau){\rm d}\tau\,.
\eeq{eq:SG}
From well known results on the heat semigroup (e.g. \cite[Sections 2.6-2.7]{Lunardi}) we have 
$$
\|D_x^\ell {\rm e}^{(t-\tau)\Delta} F_n\|_{H^{k,r}(\mathbb R^3)}\le c\, (t-\tau)^{-\frac{\ell}{2}}\|F_n\|_{H^{k,r}(\mathbb R^3)}\,.
$$
Employing in the latter the generalized Young inequality for convolutions, we prove \Eqref{B3} and \Eqref{St} for $v_n$ and $F_n$ and, consequently, obtain the desired result via a density argument. \end{proof}

\bigskip

\begin{proof}[Proof of Theorem \ref{theo:02}.] 
As is known, the space-time mollification of $\omega$, $\omega_\eta$, satisfies the equation
\be
\partial_t\omega_\eta-\Delta\omega_\eta=\Div(u \otimes \omega - \omega \otimes u)_\eta\,,\ \ \mbox{locally in $Q$}\,.
\eeq{BnT}
Moreover, $u$ obeys the (local) Helmholtz decomposition \Eqref{5}, that we rewrite more precisely here:
\be
u(x,t)=\int_{G}\nabla \mathscr E(x-y)\times \omega(y,t)+\nabla\phi(x,t)\,,\ \ \mbox{a.e. \,$(x,t)\in G$}\,,
\eeq{HD}
where $G$ is an arbitrary open set compactly contained in $Q$, $\mathscr E$ is the fundamental solution to the Laplacian on $\bb R^3$ and $\phi$ is harmonic. Therefore,
\be
\Phi(\cdot,t):=\nabla\phi(\cdot,t)\in C^\infty(G)\,. 
\eeq{har}
The works of Serrin \cite{Serrin} and Takahashi \cite{Taka}  guarantee that \Eqref{B111} combined with \Eqref{BnT} entails
\be
\omega\in L^\beta(L^\infty)\,,\ \ \mbox{for all $\beta\in (1,\infty)$}\,.
\eeq{Bn1}
From \Eqref{HD},  \Eqref{har} and the assumption \Eqref{B111} it then follows that 
$$
u\in L^{s'}(L^\infty)\,,
$$
which, along with \Eqref{Bn1}, entails
\be
(u \otimes \omega - \omega \otimes u )\in L^q(L^\infty)\, \quad q \in (2, s').
\eeq{Bn2}
Let $S=G\times(t_1,t_2)$ be compactly contained in $Q$, and let $\psi_S\in C_c^\infty(Q)$ be a smooth cut-off function so that $\psi_S\equiv1$ on $S$, and $0\leq\psi_S\leq1$. Consider the problem 
\be
\partial_tv-\Delta v=\psi_S\Div(u \otimes \omega - \omega \otimes u)_\eta\,,\ \ \mbox{in $\R^3\times(0,T)$}\,,\ \ v(x,0)=0\,.
\eeq{Bn3}
By \Eqref{Bn2} and \cite[Corollary 4.2]{GGS}
we then infer that \Eqref{Bn3} has a unique solution $v\in H^{\frac12,q}(L^r)\cap L^q(H^{1,r})$. On the other hand, from \Eqref{BnT} it follows that
\be
v(x,t)=\omega_\eta(x,t)+K(x,t)\,,\, \ \ (x,t)\in S\,,
\eeq{reg}
where $K$ is a solution to the heat equation in $S$. Then \Eqref{reg} together with \Eqref{Bn3} furnish, in particular,
\be
\omega \in L^q(H^{1,r})\,, \ \ \text{on } S \text{ and for arbitrary $r\in (1,\infty)$}\,.
\eeq{ome} 
The latter combined with \Eqref{HD}, implies 
\be
u\in L^q(H^{2,r})\,, \ \ \text{on } S \text{ and for arbitrary $r\in (1,\infty)$}\,.
\eeq{un}
Setting
\be
F:=\psi_S\,\Div (u \otimes \omega - \omega \otimes u)\,,
\eeq{EF}
from \Eqref{ome} and \Eqref{un} it follows 
\be
F\in L^{\frac{q}2}(L^r)\,, \ \ \text{on } S \text{ and for arbitrary $r\in (1,\infty)$}\,.
\eeq{ef}
so that, applying Lemma \ref{lem:Bes}  with $k=0$ and $\sigma = q/2$ to \Eqref{Bn3},  and using \Eqref{reg} we show
$$
\omega\in L^q(H^{\ell,r})\,\  \ \ \text{on } S, \mbox{with }\ \frac1q= 1 - \frac\ell2\,,\ \ \mbox{for all}\ r\in (1,\infty)\,.
$$
Now, recalling that $q > 2$, for a given $\ell\in (1,2)$ we can choose $r>3$ such that
$k:=\ell-1>3/r$. This implies, on the one hand, that $H^{k,r}$ is an algebra and, on the other hand, that
\be
\omega\in L^q(H^{k+1,r})\, \   \ \text{on } S.
\eeq{chia}
We next estimate $F$ in $H^{k,r}$. To this end, we observe that, from \Eqref{HD} and \Eqref{har} we deduce that locally in $G$,
\be
\|u\|_{{k+1,r}}\le c(\|\omega\|_{k,r}+\|\Phi\|_{k+1,r})\,.
\eeq{P2}
Since $H^{k,r}$ is an algebra, we employ  \Eqref{P2} to get
\be
\|uD_x\omega\|_{k,r}\le c(\|u\|_{k,r}\|\omega\|_{k+1,r})\le c(\|\omega\|_{k+1,r}^2+\|\Phi\|_{k+1,r}^2)\,,
\eeq{P3}
and
\be
\|\omega D_xu\|_{k,r}\le c(\|\omega\|_{k,r}\|u\|_{k+1,r})\le c (\|\omega\|_{k+1,r}^2+\|\Phi\|_{k+1,r}^2)\,,
\eeq{P4}
locally, in $G$. Observing that
$$
\|F\|_{k,r}\le c(\|uD_x\omega\|_{k,r}+\|\omega D_xu\|_{k,r})\,,
$$
and that $\Phi$ has the same time-summability as $u$, 
from \Eqref{chia}, \Eqref{P3} and \Eqref{P4} we then conclude
$$
F\in L^{\frac{q}2}(H^{k,r})\, \ \ \text{on } S.
$$
We may then apply Lemma \ref{lem:Bes}  to \Eqref{Bn3}--\Eqref{EF} by choosing $p = q$  and deduce, with the help of \Eqref{reg},
\be
\omega\in L^q(H^{k + \ell,r})\,,\ \ \text{on compact subsets of }S, \text{ where } \frac1q =1-\frac{\ell}2. 
\eeq{fine}
Thus, for given $q >2$,   there is $\ell > 1$ depending only on $q$, such that  \Eqref{chia} implies \Eqref{fine}. The property stated in the theorem then follows by a straightforward bootstrap argument.
\end{proof}

\begin{remark}  It is worth observing that, if the condition on $s'$ and $s$ in \Eqref{B111} is strengthened to
\[
\frac{2}{s'}+\frac{3}{s}<1,
\]
as originally assumed by Serrin \cite{Serrin}, then the choice $q=s'$ becomes admissible. Indeed, the stronger inequality allows one to take $\beta=\infty$ in \Eqref{Bn1} \cite{Serrin}, and hence $q=s'$ in \Eqref{Bn2}.  
\end{remark}

\setcounter{equation}{0}
\section{Relaxing the assumption $\omega\in L^{2,2}$ and proof of Theorem \ref{theo:maingen}.} The purpose of this section is to demonstrate that the assumption on $\omega$ in Theorem~\ref{theo:02} may be substantially relaxed and, under suitable conditions, removed entirely. This yields Theorem~\ref{theo:maingen} in its full generality and thereby completes its proof.

\begin{theorem}\label{theo:main2} 
Let $u$ be a distributional solution to the Navier-Stokes equations in $Q:=B\times(0,T)$, satisfying
\be
u\in L^{s^\prime,s}\,,  \ \ \frac2{s^\prime}+\frac3{s}=1\,,\ \ s'\in (2,\infty)\,. 
\eeq{SCthm}
Then, if $s'\ge 4$, necessarily $\omega\in L^{2,2}$. The same conclusion holds if $s'\in (2,4)$ provided $u\in L^{4,p}$ for some $p>1$.
\end{theorem}

The proof of the theorem will be an immediate corollary to several lemmas that we are going to show next.\par
We begin by recalling the Oseen fundamental tensor solution, $
{I}\!\!\varGamma = \{\varGamma_{ij}\}$, to the Stokes equation \cite[\S\, VIII.3]{Gab}, defined as 
\be\ba{ll}\medskip
\varGamma_{ij}(x-y,t-\tau)=-\delta_{ij}\Delta\Psi(|x-y|,t-\tau)+\partial_{y_i}\partial_{y_j}\Psi(|x-y|,t-\tau)=:\mathcal F_{ij}(\Psi)\,,\\
\Psi(r,s):=\left\{\ba{ll}\medskip\Frac{1}{4\pi^{3/2}s^{1/2}}\Frac{1}{r}\Int 0{r}{{\rm e}^{-\frac{\rho^2}{4s}}}{\rm d}\rho\ \ &\mbox{if $s>0$}\\
0\ \ &\mbox{if $s\le 0$}\ea\right.
\ea
\eeq{Bruno}
and satisfying
\be\left.\ba{ll}\medskip\partial_\tau\varGamma_{ij}+\Delta_y\varGamma_{ij}=0\\
\partial_{y_j}{\varGamma_{ij}}=0\ea\right\}\ \ \mbox{$ \tau<t$}\,.
\eeq{GA}
Also, if $v\in L^p(B)$ , $p\in(1,\infty)$, by the Helmholtz decomposition  \cite[Theorem III.1.2]{Gab}, we have, 
$$
v=v_\sigma+\nabla\phi\,,\ \ \|\nabla \phi\|_p\le c\,\|v\|_p\,.
$$
If, in addition,  $\Div v=0$,  then $\phi$ is harmonic, and by standard elliptic regularity, we deduce for any $B'\Subset B$  
\be
\max_{x\in B'}|D_x^k\phi(x)|\le C\,\|v\|_p\,,\ \ \mbox{all $|k|\ge  1$,\ $p\in(1,\infty)$}\,.
\eeq{Hd}

\begin{lemma}\label{lem:HE}  
Let $Q:=B\times (0,T)$, and let  $v\in L^{1,p}(Q)$, $p>1$, with $\Div v(t)=0$ in $B$, a.a. $t\in (0,T)$, in the sense of distributions. Then, if  
\be
\int_Q {v}\cdot(\partial_t\varphi+\Delta\varphi)\,{\rm d}x{\rm d}t=0\ \ \mbox{for all $\varphi\in\cald(Q)$}\,,
\eeq{C1}
it follows that $v _\sigma\in C^\infty(Q')$ for any $Q':=B'\times(t_1^\prime,t_2^\prime)\Subset Q$. Moreover,  there is $C=C(Q',Q,\alpha,\beta)>0$ such that 
$$
\sup_{(x,t)\in Q'}|D^{\alpha}_tD^{|\beta|}_x v _\sigma(x,t)|\le C\,\|v _\sigma\|_{L^{1,1}(Q)}\ \ \ \mbox{for all $\alpha,|\beta|\ge 0$\,.}  
$$
\end{lemma}

\begin{proof}
Clearly, \Eqref{C1} is equivalent to
\be
\int_Q {v_\sigma}\cdot(\partial_t\varphi+\Delta\varphi)\,{\rm d}x{\rm d}t=0\,.
\eeq{C2}
Let $Q^{\prime\prime}:=B^{\prime\prime}\times (t_1,t_2)$ be such that 
$${Q'}\Subset Q^{\prime\prime}\Subset  Q,$$ and  let  $\varphi$ be an arbitrary element of $\cald(Q{''})$. For $\eta$ sufficiently small, the space-time mollification, $\varphi_\eta$, of $\varphi$ belongs to $\cald(Q)$ and can  therefore be replaced in \Eqref{C2}. By a standard argument we then show that the mollifier, $v _{\sigma\eta}$, of $v _\sigma$ satisfies
\be\left.\ba{ll}\medskip
\partial_tv_{\sigma \eta}-\Delta v_{\sigma \eta}=\nabla p^{\eta}\\
\Div v_{\sigma\eta}=0\ea\right\}
\ \ \mbox{in $Q''$}\,, 
\eeq{C3}
for some smooth, harmonic scalar function $p^{\eta}$.
Let $\delta=\dist\{\partial Q',\partial Q^{\prime\prime}\}$.
For $(x,t)\in Q'$, we denote by   $\varphi^\delta=\varphi^\delta(x-y,t-\tau)$ a smooth function that is equal to 1 if $|x-y|\le\delta/2$ and $|t-\tau|\le\delta/2$, while it is equal to 0 if either $|x-y|\ge\delta$ or $|t-\tau|\ge\delta$. Define
$$
\varGamma_{ij}^\delta(x-y,t-\tau):=\mathcal F_{ij}(\varphi^\delta\,\Psi)\,,
$$
where the function $\mathcal F$ is defined in \Eqref{Bruno}$_1$. From \Eqref{GA} and the properties of $\varphi^\delta$ we readily deduce that
\be\left.\ba{ll}\medskip\partial_\tau\varGamma_{ij}^\delta+\Delta_y\varGamma_{ij}^\delta=H_{ij}^\delta(x-y,t-\tau)\\
\pde{\varGamma_{ij}^\delta}{y_j}=0\ea\right\}\ \ \mbox{$ \tau<t$}\,.
\eeq{GAd}
where 
$$
H_{ij}^\delta=H_{ij}^\delta(z,s)
\ \ \mbox{is a $C^\infty$ function with $\supp (H_{ij}^\delta)\subset \{\delta/2<|z|+s<\delta\}$}\,.
$$ 
Set ${\varGamma}_i^\delta=(\varGamma_{i1}^\delta,\varGamma_{i2}^\delta,\varGamma_{i3}^\delta)$,
and let
$$
\mathbb T(u,p)=-p\,\mathbb I+\nabla u+(\nabla u)^\top
$$
with $\mathbb I$ identity matrix, 
denote the Cauchy stress tensor.
Then, for small $\varepsilon>0$ and $t\in(t_1+\varepsilon,t_2)$, we use the well-known Green's formula  
$$\ba{rl}\medskip
\Int{t_1}{t-\varepsilon}&\!\!\!\!\Int{B^{\prime\prime}}{}\{(\partial_\tau{ u_1}+\Delta  u_1-\nabla p_1)\cdot u_2+
(\partial_\tau{ u_2}-\Delta  u_2+\nabla p_2)\cdot u_1\}{\rm d}y\,{\rm d}\tau
\\
&=\Int{t_1}{t-\varepsilon}\Int{\partial B^{\prime\prime}}{}\big[ u_1\cdot\mathbb T( u_2,p_2)- u_2\cdot\mathbb T( u_1,p_1)\big]\cdot n {\rm d}\sigma_y+
\left.\Int{B^{\prime\prime}}{} u_1(y,\tau)\cdot u_2(y,\tau){\rm d}y\right|_{\tau=t_1}^{\tau=t-\varepsilon}\,,
\ea
$$
with $ u_1:={\varGamma}_i^\delta(x-y,t-\tau)$, $i=1,2,3$, $p_1\equiv0$, and $u_2:=v _{\sigma\eta}(y,\tau)$, $p_2=p^\eta$. 
Employing \Eqref{C3} and \Eqref{GAd}, and observing that, by the properties of $\varphi^\delta$,
$$\ba{ll}\medskip
\partial_x^k\varGamma^\delta_i(x-y,t-\tau)=0\,,\ \ (x,t)\in Q'\,, \ \,(y,\tau)\in \partial B^{\prime\prime}\times (t_1,t-\varepsilon)\,,\ \     |k|\ge 0\,,\ i=1,2,3\,,\\ \varGamma_i^\delta(x-y,t_1)=0\,,\ \ (x,y)\in B'\times B^{\prime\prime}\,, \ \, i=1,2,3\,,
\ea
$$
we then show
\be
\int_{B^{\prime\prime}}(v_{\sigma\eta})_j(y,t-\varepsilon)\varGamma_{ij}^\delta(x-y,\varepsilon){\rm d}y=\int_{t_1}^{t-\varepsilon}\int_{B^{\prime\prime}}H_{ij}^\delta(x-y,t-\tau)(v_{\sigma\eta})_j(y,\tau){\rm d}y\,{\rm d}\tau\,.
\eeq{1.2_n}
By \cite[\S\,5{\small 3}]{Oseen} (see also \cite[Lemma VIII.3.1]{Gab}) and the property of $\varphi^\delta$ we infer
\be\ba{rl}\medskip
\Lim{\varepsilon\to0}\Int{ B^{\prime\prime}}{}&\!\!\!\!(v _{\sigma\eta})_j(y,t-\varepsilon)\varGamma_{ij}^\delta(x-y,\varepsilon){\rm d}y\\ \medskip
&\!\!\!\!=(v _{\sigma\eta})_j(x,t)-\Frac1{4\pi}\Int{\partial B^{\prime\prime}}{}\Frac{1}{|x-y|}\big[\varphi^\delta(y,t)\Frac{x_j-y_j}{|x-y|^2}+\partial_{y_j}\varphi^\delta(y,t)\big]\,v _{\sigma\eta}(y,t)\cdot n\,{\rm d}\sigma_y\\
&\!\!\!\!=(v _{\sigma\eta})_j(x,t)
\,.\ea
\eeq{eps}
Thus, letting $\varepsilon\to0$ in \Eqref{1.2_n} and using \Eqref{eps}, and then letting $\eta\to 0$ in the resulting relation, we easily deduce (after redefining $v_\sigma$ on a set of zero Lebesgue measure)
$$
v_{\sigma i}(x,t)=\int_{t_1}^t\int_{B^{\prime\prime}}H^\delta_{ij}(x-y,t-\tau)v_{\sigma j}(y,\tau){\rm d}y\,{\rm d}\tau\,,\ \ \ (x,t)\in Q'\,,
$$
from which the lemma follows.
\end{proof}

\begin{lemma}\label{lem:Oseen}
Let $\mathscr R_{t_1,t_2}:=\R^3\times (t_1,t_2)$, and let $\mathbb{F}=\{F_{ij}\}$ be a  sufficiently smooth tensor function in  (e.g., $\mathbb{F}$ is continuous in $\mathscr R_{t_1,t_2}$ and H\"older continuous in $x\in \R^3$ with its first derivatives, uniformly in $t$), with a bounded spatial support independent of $t$. Then, the vector function $v $ with components
\be
v_i(x,t)=\int_{t_1}^t\int_{\R^3}\partial_{y_\ell}\varGamma_{ij}(x-y,t-\tau)\cdot F_{\ell j}(y,\tau)dy\,d\tau\,,\ \ (x,t)\in \mathscr R_{t_1,t_2}\,,\ \ i=1,2,3,
\eeq{SS}
is smooth as well (e.g., twice differentiable in the space variables and differentiable in time with all derivatives continuous in $\mathscr R_{t_1,t_2}$) and 
satisfies the equations
\be\left.\ba{ll}\medskip
\partial_t{v }=\Delta v +\nabla\Phi+\Div\mathbb F\\
\Div v =0\ea\right\}\ \ \mbox{ in $\mathscr R_{t_1,t_2}$}
\eeq{SS1}
for a suitable smooth scalar field $\Phi$ (e.g., differentiable in space with each derivative continuous in $\mathscr R_{t_1,t_2}$). 
\par
Moreover, suppose
\be 
\mathbb F\in L^{\frac{q'}2,\frac q2}(\mathscr R_{t_1,t_2})\,,\ \ \frac2{q'}+\frac3q=1\,.
\eeq{BM1}
Then
$$
v \in L^{r',r}(\mathscr R_{t_1,t_2})
$$
where
\be
\frac2{r'}+\frac3r=1\,,\  \ \ 2r^\prime> q^\prime \ge\frac{4r^\prime}{2+r'}\,,
\eeq{BM2}
and 
\be
\|v \|_{L^{r',r}}\le c\,\|\mathbb F\|_{L^{\frac{q'}2,\frac q2} }
\eeq{vf}
\end{lemma}
\begin{proof} 
The first part of the lemma is a classical result of Oseen (see \cite[\S\S 5-7]{Oseen}).
To show the second part, we begin to recall the following estimate of the fundamental tensor, also due to Oseen \cite[p. 70]{Oseen},
\be
|D_{\xi}{I}\!\!\varGamma(\xi,s)|\le \Frac{C}{(|\xi|^2+s)^2}\,,
\eeq{SS_0}
for some constant $C>0$. Thus, taking into account \Eqref{SS_0}, we apply Young's inequality for convolution in \Eqref{SS} to get 
\be
\|v (t)\|_{r}\le \int_{t_1}^{t}(t-\tau)^{-2+\frac{3}{2p}}\|\mathbb F(\tau)\|_{\frac q2}\,{\rm d}\tau\,,
\eeq{FN_0} 
where
\be
\frac 1r=\frac 1p+\frac 2q-1\,
\,.
\eeq{FN}
Next, we observe that the function $h(s)=s^{-\alpha}$, $\alpha:=2-\frac3{2p}$, is in the weak space $L_w^{1/\alpha}(0,T)$, $T>0$ for all $\alpha>0$. In particular, if $\alpha<1$, namely, 
\be
p\in[1,\mbox{$\frac32$})\,.
\eeq{fn}
we may use in \Eqref{FN_0}  the generalized Young inequality 
to obtain
$$
\|v \|_{L^{r',r}}\le c_1\|\mathbb F\|_{L^{\frac{q'}2,\frac q2} }
$$
where
\be
\frac1{r'}=\alpha+\frac2{q'}-1=1+\frac2{q'}-\frac3{2p}\,.
\eeq{FN2}
Therefore, from \Eqref{BM1}$_2$, \Eqref{FN} and \Eqref{FN2}, we conclude the validity of \Eqref{BM2}$_1$. Moreover,  from \Eqref{FN2} and \Eqref{fn},
we infer $2r^\prime> q^\prime$. Finally, from \Eqref{FN} it follows $1/r\le 2/q$ which, after using \Eqref {BM1}$_2$ and \Eqref{BM2}$_1$, furnishes
the second inequality in \Eqref{BM2}$_2$,
thus completing the proof of the lemma.
\end{proof}

\begin{lemma}\label{lem:vorticity} 
Let $u$ be a distributional solution to the Navier-Stokes equations in $Q:=B\times(0,T)$ satisfying 
\be
u\in L^{s^\prime,s}\,,  \ \ \frac2{s^\prime}+\frac3{s}=1\,,\ \ \mbox{$s^\prime\in [4,8]$.}
\eeq{SC}
Then $\omega\in L^{2,2}$. 
\end{lemma}
\begin{proof} 
Let $Q^\prime:=B'\times (t_1,t_2)\Subset Q$. Then the space-time mollifier, $u_\eta$, of $u$ satisfies the following equation
\be\left.\ba{ll}\medskip 
\partial_t u_{\eta}- \Delta u_\eta=-\Div( u u)_\eta-\nabla p^{\eta}\\
\Div u_\eta=0\ea\right\}
\ \ \mbox{in $Q'$}\,,
\eeq{3.5}
for some smooth scalar field $p^{(\eta)}$.
Consider, next, the following linear problem
\be\ba{cc}\medskip\left.\ba{ll}\medskip
\partial_t v- \Delta v=-\Div( u u)_\eta-\nabla \phi\\
\Div v=0
\ea\right\}\ \ \mbox{in $Q'$\,,}\\
v(x,t_1)=0\,,\ \ x\in B'\,;\ \  v(x,t)=0 \ \ \ \mbox{$(x,t)\in\partial B'\times (t_1,t_2)$}\,.
\ea
\eeq{3.6}
It is well known that this problem has one (and only one) regular solution $v^\eta$  (for instance, 
$v_\eta\in W^{1,2}((t_1,t_2); L^2(B'))\cap L^2((t_1,t_2);W^{2,2}(B')))$ which satisfies, in particular, the ``energy inequality''
\be
\|v^\eta\|_{L^{\infty,2}(Q')}+\|\nabla v^\eta\|_{L^{2,2}(Q')}\le C\,\|(u u)_\eta\|_{L^{2,2}(Q)}\,.
\eeq{3.7}
Now, from \Eqref{3.5} and \Eqref{3.6}$_1$ it follows that $ w_\eta:= u_\eta-v^\eta$ solves the Stokes equations
\be\left.\ba{ll}\medskip
\partial_t w_\eta-\Delta w_\eta=-\nabla q^{\eta}\\
\Div w_\eta=0\ea\right\}
\ \ \mbox{in $Q'$}\,.
\eeq{3.8}
By Lemma \ref{lem:HE}  we thus deduce, in particular,
$$
\|\nabla u_{\sigma\eta}\|_{L^{2,2}(Q^{\prime\prime})}\le C(\| u_\eta\|_{L^{1,1}(Q')}+\|\nabla v^\eta\|_{L^{2,2}(Q')})\,,  
$$
for all $Q^{\prime\prime}\Subset Q'$.
We then observe that $L^{s',s}\subset L^{4,4}$, for $s'\in[4,8]$. Therefore, 
since  $\|( u
u)_\eta\|_{L^{2,2}(Q)}\le C\,\| u\|_{L^{4,4}(Q)}^2$, we conclude that there exists $C_0>0$  such that
$$
\|\nabla u_{\sigma}\|_{L^{2,2}(Q^{\prime\prime})}\le C_0\,.
$$
However, by \Eqref{Hd} and the assumption, we have $D^2_x\phi\in L^{2,2}$, which implies $\nabla u\in L^{2,2}$, thus proving the stated property.
\end{proof}

\begin{lemma}\label{lem:velocity} 
Suppose $u$ is a distributional solution to the Navier-Stokes equations in $Q:=B\times(0,T)$, satisfying
\be
u\in L^{q^\prime,q}\,, \ \ \ \frac2{q^\prime}+\frac3{q}=1\,,
\eeq{SC1}
for some $q^\prime\in (2,4)\cup (8,\infty)$. Assume also that $u\in L^{4,p}$, some $p>1$, if $q'\in (2,4)$. Then, there exists $s^\prime=s^\prime(q^\prime)$ and corresponding $s$ such that \Eqref{SC} holds.
\end{lemma}
\begin{proof}  
Let $Q:=B\times(t_1,t_2)$,  and let $\zeta=\zeta(x)$ be a smooth function that is 1 on arbitrarily given $B^{\prime\prime}\Subset B$ and 0 in $\R^3\backslash B$. Next, consider the vector function $v$ in \Eqref{SS} with $F_{\ell j}\equiv-\zeta\,(u_\ell u_j)_\eta$. By Lemma \ref{lem:Oseen}, we then deduce, on the one hand, that $ v\equiv v^\eta$ satisfies
\be\left.\ba{ll}\medskip
\partial_t{v}=\Delta v+\nabla\Phi-\Div[\zeta\, ( u u)_\eta]\\
\Div v=0\ea\right\}\ \ \mbox{ in $\mathscr R_{t_1,t_2}$},
\eeq{3.9}
and also, by \Eqref{vf} and the  assumption, that
\be
\|v^\eta\|_{L^{r',r}(Q)}\le c\,\|(u u)_\eta\|_{L^{\frac{q'}2,\frac q2}(Q) }\le c\,\|u\|_{L^{{q'},q}(Q)}^2 
\eeq{3.10}
where
\be\ba{ll}\medskip\displaystyle
2r'>q'\ge\frac{4r'}{2+r'}\,,\\  \displaystyle\frac{2}{r'}+\frac{3}r=1\,.
\ea
\eeq{3.11}
From \Eqref{3.5} and \Eqref{3.9} it then follows that $w_\eta:=u_\eta-v^\eta$ satisfies \Eqref{3.8}  in $Q^{\prime\prime}=B^{\prime\prime}\times(t_1,t_2)$, where $\zeta(x)\equiv 1$. Therefore, employing Lemma \ref{lem:HE} 
and \Eqref{3.10}  we deduce
$$\ba{rl}\medskip
\|u_{\sigma\eta}\|_{L^{r',r}(Q')}&\!\!\!\!\le c\,(\|u_{\sigma\eta}\|_{L^{1,1}(Q)}+\|(u u)_\eta\|_{L^{\frac{q'}2,\frac q2}(Q) })\\
&\!\!\!\!\le c\,(\| u\|_{L^{1,2}(Q)}+\|u\|_{L^{{q'},q}(Q)}^2)\,,\ \ \mbox{arbitrary $Q'\Subset Q$.}\ea 
$$
From this relation we then infer
\be
\|u_\sigma\|_{L^{r',r}(Q')}
\le c\,(\|u\|_{L^{1,2}(Q)}+\|u\|_{L^{{q'},q}(Q)}^2)\,,\ \ \mbox{arbitrary $Q'\Subset Q$.}
\eeq{3.12}
If $q^\prime\equiv q^\prime_0>8$, then $q^\prime\ge 4r'/(2+r')$ for any admissible choice of $r'$. In view of \Eqref{3.11} it then follows that 
\be
u_\sigma\in L^{r',r}\,,\,  \ \ \frac2{r'}+\frac3{r}=1\,,  
\eeq{JM}
for all $r'>q^\prime_0/2$. We want to show, by a recurrence procedure that uses \Eqref{3.12}, the existence of $\bar{r}'\in [4,8]$ for which \Eqref{JM} holds. In each step of this procedure, the corresponding value of $r'$ will be less than $q^\prime_0$, so that by \Eqref{HD} and \Eqref{har}, we also get, at each step, $\nabla\phi\in  L^{r',r}$. Consequently, \Eqref{JM} continues to hold with $u_\sigma$ replaced by $u$.   
Now, suppose \Eqref{SC1} holds for  some $q^\prime_0>8$. Then,   it follows that  \Eqref{JM} holds also for any $r'\ge \frac23q'_0:=q'_1$. If $q'_1\le 8$, we are done, otherwise we use \Eqref{3.12} with $q^\prime=q^\prime_1$ and derive, again from \Eqref{3.11}, the validity of \Eqref{JM} for any $r'\ge \frac23q'_1=(\frac23)^2q'_0:=q'_2$. If $q'_2\le 8$ the proof is complete; if not, we iterate this process a finite number of times, until we find an integer $\bar{n}$ such that condition \Eqref{JM} is satisfied for all $r'\ge q^\prime_{\bar{n}}:=(\frac23)^{\bar{n}}q^\prime_0\le 8$ and, as a result, by some $\bar{r}^\prime\in[4,8]$. Note that, in this process, the domains $Q$ and $Q'$ may ``shrink'' at each step by an arbitrarily given small amount, and the constant $c$ in \Eqref{3.12} may also change. However, $Q'$ and $Q$ are themselves arbitrary and the number of steps is finite. Therefore,  the proof in the case $q>8$ is completed. 
Suppose next $q^\prime\in (2,4)$. We  notice that since in this case, by assumption, $u\in L^{4,p}$ for some $p>1$, Thus, by \Eqref{Hd} it follows  $\nabla\phi\in L^{4,{\rho}}$, for all ${\rho}\in (1,\infty)$. Consequently, we infer that 
\be
u\in L^{\gamma,\rho}\,\  \gamma\le4, \ \ \rho\in (1,\infty),\,\ \mbox{as long as $u_\sigma$ satisfies the same property}\,.
\eeq{im} 
Next, we write $q'=4\,\alpha$ for some $\alpha\in (\frac12,1)$. From \Eqref{3.11}$_1$ we find that there exists $ r'_0$ such that \Eqref{JM} holds with
$$
\frac1{r'_0}=\frac1{2\alpha}-\frac12\,, \ \ \mbox{and all $r'\le r^\prime_0$.}
$$
If $r^\prime_0>4$, in view of \Eqref{im},  we are done. Otherwise, again by \Eqref{im}, we can use this value of $r'$ (and corresponding $r$) as $q'$ (and $q$) in \Eqref{3.12}. By \Eqref{3.11}$_1$ we then deduce that there is $r'\equiv r'_1$
such that \Eqref{JM} is satisfied with
$$
\frac1{r'_1}=\frac1{\alpha}-1-\frac12\,,  \ \ \mbox{and all $r'\le r^\prime_1$.}
$$
If $r^\prime_1<4$ we repeat this procedure $n+1$ times, till we find $r'_n>4$ such that \Eqref{JM} is valid with $r'_n$ and all $r'\le r'_n$. To show that, for any given $\alpha\in(\frac12,1)$, such $r'_n$ exists, we observe that, by the above procedure, it follows that 
$$
\frac1{r'_n}=\frac{2^{n-1}}{\alpha}-\sum_{k=0}^{n-1}2^k-\frac12\,,\ \ n\ge1.
$$
Thus, the request $r'_n>4$ for some $n=\bar{n}$ is equivalent to the existence of $\bar{n}$ such that
$$
\frac1\alpha<\sum_{k=0}^{\bar{n}-1}2^{-k}+\frac3{2^{\bar{n}+1}}\,,
$$
which is certainly satisfied since $1/\alpha<2$ and the series converges to 2. As in the case $q>8$ discussed before, also in this iterative scheme,  $Q'$ may ``shrink'' at each step by an arbitrarily given small amount as well as the constant $c$ in \Eqref{3.12} may change but,  as previously mentioned, this does not affect the conclusion of the lemma which is thus completely proven.
\end{proof}

Theorem~\ref{theo:main2} follows at once from the previous two lemmas. Consequently, combining Theorem~\ref{theo:main2} with Theorem~\ref{theo:02}, we arrive at our main result, Theorem~\ref{theo:maingen}.

\ \\

\noindent{\bf Acknowledgement.} The work of R.M. Chen is partially supported by NSF grant DMS-2205910. The work of G.P. Galdi is partially supported by NSF grant DMS-2307811.     
The work of A.J. Schikorra is partially supported by NSF Career DMS-2044898

\end{document}